\documentclass[12pt]{amsart}
\usepackage{amssymb,latexsym, amsmath, amsxtra}
\usepackage{graphicx}
\usepackage{subfig}
 \newenvironment{dedication}
        {\vspace{1ex}\begin{quotation}\begin{center}\begin{em}}
        {\par\end{em}\end{center}\end{quotation}}
\newtheorem{conjecture}{Conjecture}
\newtheorem{theorem}{Theorem}
\newtheorem{lemma}{Lemma}
\newtheorem{corollary}{Corollary}
\newtheorem{proposition}{Proposition}
\newcommand{\legendre}[2]{\genfrac{(}{)}{}{}{#1}{#2}}
\newtheorem{remark}{Remark}
\begin{document}

\title{Character Sums and the Riemann Hypothesis}
\author{Brian Conrey}
 \thanks{Research supported by an FRG grant from NSF}
 
\maketitle
 
  \begin{dedication}
 Dedicated to Henryk on his semisesquicentennial 
 \end{dedication}

 \begin{abstract}
 We prove that an innocent looking inequality implies the Riemann Hypothesis and show a way to approach this inequality through sums of Legendre symbols.
 \end{abstract}
 
\parindent=0cm
\parskip=.2cm

Let 
$$f(x)=\sum_{n=1}^\infty \frac {\lambda(n)\sin 2\pi n x}{n^2}$$
where $\lambda$ is the Liouville lambda-function\footnote{$\lambda$ is completely multiplicative and takes the value $-1$ on primes so that $\lambda(p_1^{e_1}\dots p_r^{e_r} )=(-1)^{e_1+\dots+e_r}.$}.  Since $|\lambda(n)|=1$, this series is absolutely convergent for real $x$, so that $f$ is continuous, odd and periodic with period 1 on $\mathbb R$.
Here is a plot of $f(x)$ for $0\leqslant x\leqslant 1$ using 1000 terms of the series defining $f$:
\begin{center}
 \includegraphics[scale=0.9]{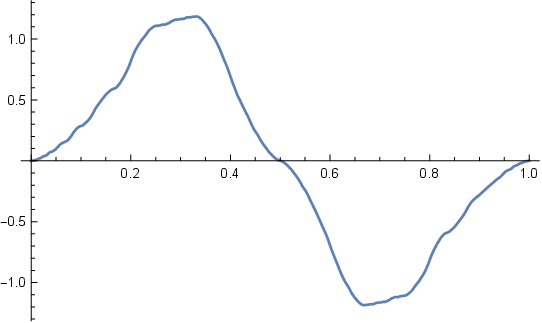}
\end{center}

\begin{theorem}
If  $f(x)\geqslant 0$ for $0\leqslant x \leqslant 1/4$, then the Riemann Hypothesis is true.
\end{theorem}
 
 Theorem 1 is deceptive in that it looks like it should be a a simple matter to prove that $f(x)$ is non-negative. A problem is that it is not clear  whether 
$f(x)$ is differentiable or not, and even if it is it would be difficult to estimate the derivative. So, proving that $f(x)>0$ at some point doesn't immediately tell us about $f(x)$ at nearby points. 

The ``$1/4$''
in Theorem 1 can be replaced by any positive constant. So the real issue is trying to prove that $f(x) >0$ for small positive $x$. 

Note that 
$$\left|\sum_{n=N+1}^\infty \frac{\lambda(n)\sin 2\pi n x}{n^2}\right| < \int_N^\infty u^{-2}~du = \frac 1 N$$
so that if for some $x$ there is an $N$ such that 
\begin{eqnarray} \label{eqn:test} \sum_{n=1}^N \frac{\lambda(n)\sin 2\pi n x}{n^2} \geqslant \frac 1 N\end{eqnarray}
then, it must be the case that $f(x)>0$.  We will use this idea a little later.

We can give an   ``explicit formula'' for $f$ in terms of the zeros $\rho=\beta+i\gamma$ of $\zeta$:
\begin{theorem} Assuming the Riemann Hypothesis, 
$$f(x)=-\frac{4 \pi^2 x^{3/2}}{3 \zeta(1/2)}-\frac{8 \pi^2}{3}x^{3/2}\sum_{n\leqslant 4x} \frac{\ell(n)}{\sqrt{n}}\bigg(1-\frac{n}{4x}\bigg)^{3/2}
+\pi \lim_{T\to \infty} \sum_{\rho=1/2+i\gamma\atop |\gamma|\leqslant T}\operatornamewithlimits{Res}_{z=\rho-1} \frac{X(1-z)\zeta(2z+2)x^{1-z}}{(1-z)\zeta(z+1)}.$$
\end{theorem}
Here  $\ell(n)$ is defined through its generating function
$$\sum_{n=1}^\infty \ell(n)n^{-s}=\frac{\zeta(2s-1)}{\zeta(s)}$$
for $\Re s >1.$ Also, 
$X(s)$ is the factor from the functional equation for $\zeta(s)$ which can be defined by
$$X(s)^{-1}=X(1-s)=\frac{\zeta(1-s)}{\zeta(s)}=2 (2\pi)^{-s}\Gamma(s) \cos \tfrac {\pi s}{2}  .$$ 
Note that if the zeros of $\zeta(s)$ are simple, then the term with the sum over the zeros of $\zeta$ becomes 
$$\pi \sum_{\rho}\frac{X(2-\rho)\zeta(2\rho)x^{2-\rho}}{(2-\rho)\zeta'(\rho)}.$$

Theorem 2 is nearly a converse to Theorem 1 in the sense that if RH is true and all the zeros are simple and 
\begin{eqnarray} \label{eqn:ineq} \sum_{\rho}\bigg|\frac{X(2-\rho)\zeta(2\rho) }{(2-\rho)\zeta'(\rho)}\bigg|\leqslant -\frac{4\pi}{3\zeta(1/2)}
\end{eqnarray}
then
$f(x)\geqslant 0$
for $0\leqslant x \leqslant 1/4$.
 Note that
$$  -\frac{4\pi}{3\zeta(1/2)}=2.86834\dots$$
and
$$\sum_{|\gamma|\leqslant 1000}\bigg|\frac{X(2-\rho)\zeta(2\rho) }{(2-\rho)\zeta'(\rho)}\bigg|=
0.264954\dots$$
so that the inequality (\ref{eqn:ineq}) seems plausible.

Finally we remark that the  formula of Theorem 2 for $f(x)$ hides very well the fact that $f(x)$ is periodic with period 1!

\section{Prior results}
There has been quite a lot of work connecting partial weighted sums of the Liouville and the Riemann Hypothesis.
We refer to [BFM] for a nice description of past work. In this paper the authors prove that the smallest value of $x$ for which 
$$\sum_{n\le x}\frac{ \lambda(n)}{n} <0$$
is $x=72 185 376 951 205$.

\section{Character sums}

A possible approach to proving that $f(x)>0$ for small $x>0$ lies in the fact that $\lambda$ is completely multiplicative and takes the values $\pm 1$. This scenario resembles quadratic Dirichlet characters (for simplicity think Legendre symbols) except that Dirichlet characters can also take the value 0.    By the Chinese Remainder Theorem, for any $N$ there is a prime number $q$ such that 
$\lambda(n)=\legendre{n}{q} $ for all $n\leqslant N$ where $\legendre{.}{q}$ is the Legendre symbol\footnote{$\legendre{n}{q}=0$ if $(n,q)>1$; $\legendre{n}{q}=+1$ if $n$ is a square mod $q$; and $\legendre{n}{q}=-1$ if $n$ is not a square modulo $q$.} mod $q$.   As an example:
   $$\lambda(n)=\legendre{n}{163} $$
   for all $n\leqslant 40$, but they differ at $n=41$.
  
  Let $$f_q(x)=\sum_{n=1}^\infty \frac{\legendre{n}{q} \sin 2\pi n x}{n^2}$$
  be the Fourier sine series with $\lambda(n)$ replaced by $\legendre{n}{q}$ . If $f_q(x)\geqslant 0$  for  $0\leqslant x\leqslant 1/4$ for a sufficiently large set of $q$, then it must  also be the case that  $f(x)\geqslant 0$ for $0\leqslant x\leqslant 1/4$. 
  (The proof is that if $f(x_0)<0$ for some $0<x_0<1/4$,  then  we can find a $q$ such that $\legendre{n}{q}=\lambda(n)$ for all $n\leqslant N$ where $N$ is chosen so large that $|f(x_0)| >1/N$; then it must be the case by the analog of (\ref{eqn:test}) for $f_q$ that $f_q(x_0)<0$.) 
  The same assertion but with $q$ restricted to primes congruent to 3 mod 8 is also valid, since the Legendre symbols for these $q$ can also imitate $\lambda(n)$ for arbitrarily long stretches $1\leqslant n\leqslant N$. 
We can express this as follows:
\begin{theorem} 
If 
$$f_q(x)\geqslant 0$$
for all $0\leqslant x\leqslant 1/4$ and all primes $ q$  congruent to 3 mod 8, then the Riemann Hypothesis is true.
\end{theorem}

\begin{remark} We could just as well have stated this theorem for $q\equiv 3 \bmod 4.$ However, the intention is that we are interested in $q$ for which $\chi_q$ imitates $\lambda$. Insisting that $\chi_q(2)=-1$ leads to the condition that $q\equiv 3 \bmod 8$.  \end{remark}

 The sums $f_q(x)$ still have the same problem in that it is tricky to prove for sure that they are positive for small positive $x$. However, the 
 analogue of Theorem 2 above is much simpler, is unconditional, and leads to a straightforward way to check, for any given fixed $q$,  that $f_q(x)\geqslant 0$ for $0\leqslant x \leqslant 1/4$.
 \begin{theorem} Let $x\geqslant 0$. Let $q\equiv 3 \bmod 8$ be squarefree.  Then
$$f_{q}(x)= 2 \pi x L_q(1) -\frac{2 \pi^2 x}{\sqrt{q}}\sum_{n\leqslant xq}\legendre{n}{q}  \big( 1-\tfrac {n}{xq}\big)$$
where
$$L_q(1)=\sum_{n=1}^\infty \frac{\legendre{n}{q}}{n}.$$
\end{theorem}

 Now Dirichlet's class number formula enters the picture.  Let $K=\mathbb Q(\sqrt{-q})$ be the imaginary quadratic field obtained by adjoining $\sqrt{-q}$ to the rationals $\mathbb Q$.
 Let $h(q)$ be the class number\footnote{The class number is a measure of how close to unique factorization the integers of $K$ are; $h(q)=1$ means the integers of $K$ can be factored into primes in only one way.} of $K$. 
 Dirichlet's formula is 
 $$h(q)=\frac{\sqrt{q}}{\pi} L_q(1)$$
 for squarefree $q\equiv 3 \bmod 4$ and $q>3$; (see [D] or [IK]).
 Thus, the Theorem above can be rephrased in terms of $h(q)$.
 Moreover, we can express $L_q(1)$ as a finite character sum
 $$L_q(1)= -\frac{\pi}{q^{3/2}}\sum_{n=1}^q n\legendre{n}{q} .$$
 Since $\legendre{n}{q}$ is an odd function of $q$ we also have
  $$L_q(1)= -\frac{2 \pi}{q^{3/2}}\sum_{n=1}^{\frac{q-1}{2}} n\legendre{n}{q} $$
  and
  $$h(q)=S_q(\frac q 2 )$$
  where 
  $$S_q(N):=\sum_{n\leqslant N}\legendre{n}{q} \big(1-\tfrac n N\big)  .$$
\begin{corollary} Let $q>3$ be squarefree with $q\equiv 3 \mod 8$.  
Then
$$f_q(x)=\frac{2\pi^2x}{\sqrt{q}} \left(S_q(\frac q2)-S_q(qx)\right).$$
 \end{corollary}
Here is a plot of $f_{163}(x)=\frac{2\pi^2 x}{\sqrt{163}}(S_{163}(\frac{163 }{2})-S_{163 } (163 x))$ for $0\leqslant x\leqslant 1$ and a plot of the difference $f(x)-f_{163}(x)$:
\begin{center}
 \includegraphics[scale=0.9]{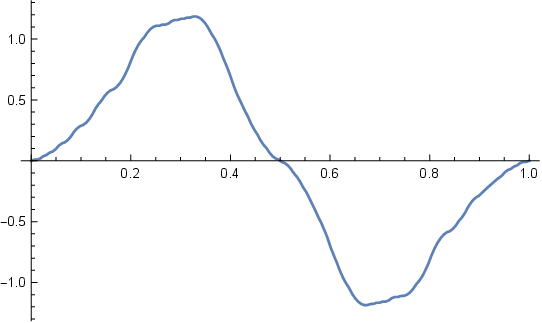}  \includegraphics[scale=0.9]{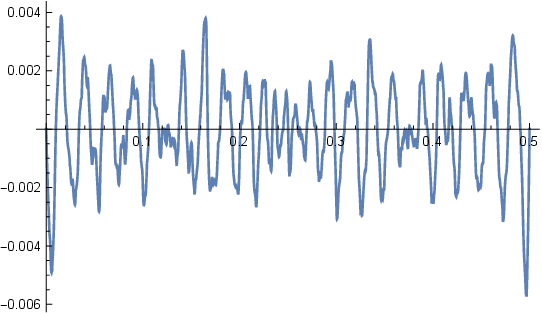}
\end{center}

We can use the corollary to prove that $f_{163}(x)\geqslant 0$ for $0\leqslant x\leqslant 1/2$  and consequently that $f(x)\geqslant 0$   for  $1/4>x\geqslant 0.043$ as follows.
\begin{eqnarray*}
f(x)&=& \sum_{n=1}^{40} \frac{\lambda(n)\sin 2 \pi n x}{n^2} +\frac \Theta {40}= \sum_{n=1}^{40} \frac{\chi_{163}(n)\sin 2 \pi n x}{n^2} +\frac \Theta {40}=f_{163}(x)+\frac \Theta {20}\\
&=& \frac{2\pi^2x}{\sqrt{163}} \left(S_{163}(\frac {163}2)-S_{163}(163x)\right)+\frac \Theta {20}
\end{eqnarray*}
where $\Theta $ denotes a number with absolute value at most 1, not necessarily the same at each occurrence.
Now for $a$ an integer, $S_{163}(163x)$ is constant for $x$ in the interval $[\frac{a}{163},\frac{a+1}{163})$. Therefore, $f_{163}(x) \geqslant \min\{f_{163}(\frac{a}{163}),f_{163}(\frac{a+1}{163})\}$ for $x$ in this interval.
We can tabulate these values:
\begin{eqnarray*} 
\begin{array}{|c|c|c|c|c|c|c|c|c|c|c|}
\hline
a&1&2&3&4&5&6&7&8&9&10\\
\hline
f_{163}(\frac{a}{163})&
  0.0095 &
  0.0095 &
  0.019 &
  0.038 &
  0.047 &
  0.066 &
  0.076 &
  0.095 &
  0.12 &
 0.14 \\
 \hline
\end{array}
\end{eqnarray*}

 Since $\frac{\Theta}{20}\leqslant .05$ it follows from (\ref{eqn:test}) that $f(x)\geqslant 0 $ for $0.25\geqslant x\geqslant \frac 7 {163}=0.043$.
\begin{corollary}
$f(x) \geqslant 0 $
for $0.043\leqslant x \leqslant 0.25$.
\end{corollary}
It seems clear that for any given $\epsilon>0$ we could replace 0.043 by $\epsilon$ in this inequality with enough computation time. 
Also, if we use Euler products instead of Dirichlet series we can show that $f(x)\geqslant 0$ for $1/4\geqslant x\geqslant 0.011$.

The following conjecture seems surprising.
\begin{conjecture}
If $q\equiv 3 \bmod 8$ is squarefree, then $f_q(x)\geqslant 0$  for $0\leqslant x\leqslant 1/2$. 
\end{conjecture}
\begin{remark} J.\ Bober has  checked that this inequality is true for all primes $q\equiv 3 \bmod 8$ up to $10^9$.
\end{remark}

Now we turn to the proofs.

\section{Useful Lemmas}
\begin{lemma}  For $y>0$ we have
$$\frac{1}{2\pi i}\int_{(c)} \frac{X(1-s) y^{1-s}}{1-s} ~ds =\frac{\sin 2 \pi y}{\pi}$$
for any $c$ satisfying $0<c< 1$ where $(c)$ denotes the path from $c-i\infty$ to $c+i\infty$
\end{lemma} 
The integrand has simple poles at $s=0,-2,-4,\dots$
with the residue at $s=-2n$ equal to 
$$\frac{1}{\pi} \frac{(-1)^n (2 \pi y)^{2n+1}}{(2n+1)!}.$$
Summing these leads to the desired formula.
See also [T]; the above is the integral of the formula (7.9.5).

\begin{lemma} If $c>0$ and $\Re a>0$, then 
\begin{eqnarray*}\frac{1}{2\pi i} \int_{(c)} \frac{\Gamma(s)\Gamma(a)}{\Gamma(s+a)} x^{-s}~ds=\left\{ \begin{array}{ll} (1-x)^{a-1} & \text{if $0<x<1$}\\
0 & \text{if $x\geqslant 1$} \end{array}  \right.
\end{eqnarray*}
 
\end{lemma} 

This formula is (7.7.14) of [T].
 
\begin{lemma}  If $c>0$, then
\begin{eqnarray*}\frac{1}{2\pi i} \int_{(c)} \frac{x^s}{s(s+1)}~ds=\left\{\begin{array}{ll} 1-\tfrac 1 x  &\text{if $x>1$}\\0& \text{if $0<x\leqslant 1$}
\end{array} \right.
\end{eqnarray*}
\end{lemma} 
This lemma is well-known and is easy to verify.

\section{Proofs of theorems}

{\it Proof of Theorem 1.} This assertion is a consequence of Landau's Theorem: ``If $g(n)\geqslant 0$ then the right-most singularity of 
$\sum_{n=1}^\infty g(n) n^{-s}$ is real.''  This is theorem 10 of [HR] and Theorem 1.7 of [MV1]. What we actually need is an integral version of this theorem:
``If $g(x)\geqslant 0$ then the right-most singularity of 
$\int_{1}^\infty g(x) x^{-s}~dx$ is real.''  The proof of this version  is essentially the same as that of the first version (see Lemma 15.1 of [MV1]).  The application to our situation is slightly subtle.
We argue as follows.
Since 
$$\sum_{n=1}^\infty\lambda(n)n^{-s}=\frac{\zeta(2s)}{\zeta(s)}$$ it follows from  Lemma 1 that
$$\frac{f(x)}{\pi}= \frac{1}{2\pi i} \int_{(c)} \frac{X(1-s)}{1-s} \frac{\zeta(2s+2)}{\zeta(s+1)} x^{1-s} ~ds$$
where $0<c<1$.  The integral is absolutely convergent for $0<c<1/2$.
By Mellin inversion we have
$$\frac{\pi X(1-s)}{1-s} \frac{\zeta(2s+2)}{\zeta(s+1)} =\int_0^\infty f(x) x^{s-2} ~dx.$$ We split the integral into two integrals at $x=4$ so that 
$$ \frac{\pi X(1-s)}{1-s} \frac{\zeta(2s+2)}{\zeta(s+1)} =\int_0^4  f(x) x^{s-2} ~dx + \int_4^\infty  f(x) x^{s-2} ~dx =I_1(s)+I_2(s),
$$ say.  The integral defining $I_1(s)$ is absolutely convergent for $\sigma >1$ and the second integral is absolutely convergent for $\sigma<1$.
Using the periodicity of $f$ we can show that the second integral converges for $\sigma<2$. Indeed, let
$$F(x)=\int_0^x f(t)~dt.$$
Then $F(n)=0$ for all integers $n$ and $F$ is bounded. Therefore,
\begin{eqnarray*}I_2(s)&=&\sum_{n=4}^\infty \int_n^{n+1} f(x) x^{s-2}~dx=\sum_{n=4}^\infty \left(\left.F(x)x^{s-2}\right|_{x=n}^{x=n+1}-(s-2)\int_n^{n+1}F(x)x^{s-3}~dx \right)\\
&=&  -(s-2)\int_4^\infty F(x)x^{s-3}~dx .
\end{eqnarray*}
This integral converges for $\Re s <2$.  So, we now have $I_2$ analytic for $\Re s <2$. Clearly,  $I_1+I_2$ is analytic for $\Re s > \max\{-1/2,\rho-1\}$ i.e. for $\Re s> 0$. (The pole of $X(1-s)$ at $s=0$ is canceled by the zero of $1/\zeta(s+1)$ at $s=0$.)
It follows that  $I_1(s)=(I_1(s)+I_2(s))-I_2(s)$ is analytic for $\Re s >0 $.    It follows that $I_2(s)$ is also analytic for $\Re s>0$, and since we already knew it was analytic for $\Re s<2$ it follows that $I_2(s)$ is entire. Now, we can write $I_1$ as 
$$I_1(s)= \int_{1/4}^\infty f(1/x)  x^{-s} ~dx. 
$$
Recall we have assumed that $f(1/x) \geqslant 0$  for $x\geqslant 4$.  Therefore, by Landau's theorem, the rightmost singularity of $I_1(s)$ is real.  Since $I_2$ is entire, it follows that the rightmost pole of $I_1(s)+I_2(s)$ must also be real. But the rightmost real pole of 
$$I_1(s)+I_2(s)= \frac{\pi X(1-s)}{1-s} \frac{\zeta(2s+2)}{\zeta(s+1)}$$
is at $s=-1/2$. This must be the rightmost pole. Therefore the poles at $\rho-1$ must all have their real parts less than or equal to -1/2. In particular,
$\Re \rho \leqslant 1/2$, which is RH.

 {\it Proof of Theorem 2.} 
We start again from 
$$
\frac{f(x)}{\pi}=\frac{1}{2\pi i}\int_{(c)}\frac{X(1-s)\zeta(2 s+2)x^{1-s}}{(1-s)\zeta(s+1)}~ds$$
where $0<c < 1/2$.
The integrand has poles only at $s=-\tfrac 1 2$ and at $s=\rho-1$ where $\rho$ is a complex  zero of $\zeta(s)$ and nowhere else in the $s$-plane.  The residue at 
$s=-\tfrac 12$ is 
$$\frac{X(\tfrac 32)}{\tfrac 32 \zeta(\tfrac12)}x^{ 3/2}=
-\frac{4 \pi  }{3 \zeta(\tfrac 12)}x^{3/2}.
$$
Assuming that the zeros are simple, the residue at $s=\rho-1$ is 
$$ \frac{X(2-\rho)\zeta(2\rho)x^{2-\rho}}{(2-\rho)\zeta'(\rho)}.$$
We (carefully) move the path of integration to $(c)$ where $-2<c<-1$. 
To do this we have to cross through a field of poles arising from the zeros of the zeta function.
To do this we use Theorem 14.16 of [T1] (see also [R])
to find a path on which   $1/\zeta(s+1)\ll T^\epsilon$ where we can safely cross.  Using the bounds  $|X(1-s)|\ll T^{\sigma-1/2}$ and $\zeta(2s+2)\ll T^{-\frac 1 2-\sigma}$
 we can get the sum of the residues arising from the zeros up to height $T$ together with an error term that tends to 0 as $T\to \infty$.
Thus, assuming the zeros are simple,
$$\frac{f(x)}{\pi}=-\frac{4 \pi x^{3/2}}{3 \zeta(1/2)} +\sum_{\rho}\frac{X(2-\rho)\zeta(2\rho)x^{2-\rho}}{(2-\rho)\zeta'(\rho)}+\frac{1}{2\pi i}\int_{(c)}
\frac{X(1-s)\zeta(2 s+2)x^{1-s}}{(1-s)\zeta(s+1)}~ds.
$$
If the zeros are not simple we modify the sum over zeros appropriately.
We make the change of variable $s\to -s$ in the integral. Using the functional equation for the $\zeta$-function and functional relations for the $\Gamma$-function, we see that the new integrand is
$$ \frac{X(1+s)\zeta(2-2 s)x^{1+s}}{(1+s)\zeta(1-s)}=-\pi^{3/2}2^{2s}\frac{\Gamma(s-\tfrac{1}{2})}{\Gamma(s+2)}\frac{\zeta(2s-1)}{\zeta(s)} x^{1+s}.
$$
By Lemma 2,
$$\frac{1}{2\pi i}\int_{(c)} \pi^{3/2}2^{2s}\frac{\Gamma(s-\tfrac{1}{2})}{\Gamma(s+2)}\frac{\zeta(2s-1)}{\zeta(s)} x^{1+s}
=\frac{8 \pi}{3}x^{3/2}\sum_{n\leqslant 4x} \frac{\ell(n)}{\sqrt{n}}\bigg(1-\frac{n}{4x}\bigg)^{3/2}.$$
Then Theorem 2 follows.
 
 {\it Proof of Theorem 4.}
 We denote $\chi_q(n)=\legendre{n}{q}$.   By Lemma 1,
\begin{eqnarray} \label{eqn:fq} f_{q}(x)=\frac{\pi}{2\pi i} \int_{(c)} L(s+1,\chi_q)X(1-s) x^{1-s} ~\frac{ds}{1-s}\end{eqnarray}
where $0<c<1$.
Since $\chi_q$ is odd, we find that the integrand has a pole at $s=0$ and nowhere else in the complex plane.
We move the path of integration to $(c)$ where $c<-1$ to see that
$$ f_{q}(x)=2\pi xL(1,\chi_q)+\frac{\pi}{2\pi i} \int_{(c)} L(s+1,\chi_q)X(1-s) x^{1-s} ~\frac{ds}{1-s}.$$
Now let $s\to -s$ in the integral and use the functional equation (see [D], [IK] or [MV1])
$$L(1-s,\chi_q)=2 q^{s-\tfrac 12}(2\pi)^{-s}\Gamma(s)\sin\tfrac {\pi s}{2}L(s,\chi_q).$$
After simplification, the integral above is
$$\frac{-2\pi^2 }{2\pi i} \int_{(c)}q^{s-\tfrac12}x^{1+s}L(s,\chi_q) ~\frac{ds}{s(s+1)}.$$
By Lemma 3, this integral is
$$ \frac{-2\pi^2 x}{\sqrt{q}}\sum_{n\leqslant xq}\chi_q(n)\bigg(1-\frac{n}{xq}\bigg).$$
The proof of  Theorem 4 is complete.
 
 \begin{remark}
 Note that the non-negativity, for $0< x < 1/4$,  of the right-hand side of (\ref{eqn:fq}) implies the Riemann Hypothesis. This  condition 
 only involves Dirichlet L-functions with quadratic characters. Thus, information solely about Dirichlet L-functions 
 potentially gives the Riemann Hypothesis. This example   shows that different L-functions somehow know about each other. 
 \end{remark}
 
 \section{Further remarks}
 
 Since 
$$h(q)\gg_\epsilon q^{1/2-\epsilon}$$
we see that 
$$f_q(x)\geqslant 0$$
for $a\ll x \ll q^{-1/2-\epsilon}$. In particular,
$$f_q(a/q) \geqslant 0$$
for $a\ll q^{1/2-\epsilon}$. But this doesn't give information about $f(x)$.
 
Also, the Polya-Vinogradov inequality tells us that
$$\max_{N}\big|\sum_{n=1}^N \chi_q(n)\big|\ll q^{1/2}\log q$$
and the work of Montgomery and Vaughan [MV]
shows that the Riemann Hypothesis for $L(s,\chi)$ implies that
$$\max_{N}\big|\sum_{n=1}^N \chi_q(n)\big|\ll q^{1/2}\log \log q.$$
Moreover, it is known that the right hand side here can not be replaced by any function that goes to infinity slower.  It is also known, assuming the Riemann Hypothesis for $L(s,\chi)$, that
$$L(1,\chi)\ll \log \log q.$$
Our desired inequality can be expressed in terms of $L(1,\chi)$ as 
\begin{eqnarray} \label{eqn:inequality}  \max_{N \leqslant \tfrac q4}\sum_{n=1}^N  \chi(n)\big(1-\frac nN\big)\leqslant \frac{\sqrt{q}}{\pi}L(1,\chi).\end{eqnarray}
It appears that both sides of this inequality can be as big as 
$\sqrt{q}\log \log q.$

 A question is whether the converse of Theorem 1 is true.  
 It might be possible to approach this by showing that the ``$\frac 3 2$'' derivative of $f(x)$ is positive at $x=0$ so that there is a small interval 
 to the right of  $0$ for which $f(x)\ge 0$. This method, or trying to prove (\ref{eqn:ineq}) directly, would involve explicit estimates (assuming RH) for $1/\zeta(s)$ in the critical strip;
 see [MV1] section  13.2 for a good approach to such explicit estimates.
 
 Finally, we mention that $f(x)$ can be evaluated at a rational 
 number  $x=a/q$ as an average involving Dirichlet L-functions $L(s,\chi)$ where $\chi$ is a character modulo $q$.

 \section{Evaluation of $f_q(a/p)$.}
 Let $p<q$ and $(a,p)=1$. We  explicitly   evaluate $f_q(a/p)$ as a sum over characters modulo $p$ as follows. We have
 \begin{eqnarray*}
 f_q(a/p)&=&\sum_{n=1}^\infty \frac {\chi_q(n) \sin \frac{2\pi a n}{p}}{n^2}
\\
&=&  \sum_{n=1}^\infty \frac {\chi_q(n) }{n^2}\frac{1}{\phi(p)} \Im \left\{\sum_{\psi\bmod p}\tau(\psi)  \overline{\psi}(an)\right\}\\
&=&  \frac{1}{\phi(p)} \Im \left\{\sum_{\psi\bmod p}\tau(\psi)  \overline{\psi}(a) 
 \sum_{n=1}^\infty \frac {\chi_q(n)\psi(n) }{n^2} \right\} \\
 &=& 
  \frac{1}{\phi(p)} \Im \left\{\sum_{\psi\bmod p}\tau(\psi)  \overline{\psi}(a)  L(2,\chi_q\overline{\psi} )\right\} .
  \end{eqnarray*}
 Now, if $\psi $ is even, then
 \begin{eqnarray*}\overline{\tau(\psi)} &=& \sum_{n=1}^p \overline{\psi(n)} e(-an/p)= \sum_{n=1}^p \overline{\psi}(-n)e(an/p) \\
 &=&  \sum_{n=1}^p \overline{\psi}(n)e(an/p)=\tau(\overline{\psi})
 \end{eqnarray*}
while if $\psi $ is odd then
 \begin{eqnarray*}\overline{\tau(\psi)}=-\tau(\overline{\psi}). \end{eqnarray*}
Thus, for even $\psi$ 
\begin{eqnarray*}
\Im \left\{ \tau(\psi)  \overline{\psi}(a)  L(2,\chi_q\overline{\psi} )+\tau(\overline{\psi})  \psi(a)  L(2,\chi_q{\psi} )
\right\} =0
\end{eqnarray*}
and for odd $\psi$
\begin{eqnarray*}
\Im \left\{ \tau(\psi)  \overline{\psi}(a)  L(2,\chi_q\overline{\psi} )+\tau(\overline{\psi})  \psi(a)  L(2,\chi_q{\psi} )
\right\} = 2 \Im\{\tau(\psi)  \overline{\psi}(a)  L(2,\chi_q\overline{\psi} )\}.
\end{eqnarray*}
Therefore,  using the fact that $\tau(\chi_p)=i\sqrt{p}$ when $p\equiv 3 \bmod 4$, we have 
\begin{eqnarray*} f_q(a/p)=\frac{1}{\phi(p)}\sum_{{\psi \bmod p\atop \psi(-1)=-1}\atop \psi^2\ne \psi_0} \Im\{ \tau(\psi)  \overline{\psi}(a)  L(2,\chi_q\overline{\psi} )\}
+\delta(p\equiv 3 \bmod 4)\frac{\sqrt{p}}{\phi(p)}\Re \{\overline{\psi}(a)L(2,\chi_q\overline{\psi})\}.
\end{eqnarray*}
We use this to prove that 
$$f_q(1/3)>0 \qquad \mbox{and} \quad f_q(1/5)>0$$
for any $q$. By the formula above we have 
\begin{eqnarray*}
f_q(1/3)=\frac{ \sqrt{3} }{2}L(2,\chi_q \chi_3) >0
\end{eqnarray*}
and 
\begin{eqnarray*}
f_q(1/5) &=& \frac{2}{\phi(5)} \Im\{(-1.17557 + 1.90211 i) L(2,\chi_q\psi_1)\}\\
&=& 1.9 \alpha -1.17 \beta
\end{eqnarray*}
where $\psi_1=\{1,i,-i,-1,0\}$ with $\tau(\psi_1)=-1.17557 + 1.90211 i$
and
  $$\alpha+i\beta =L(2,\chi_q\psi_1) = 1 + \frac{\chi_q(2) i}{2^2}-\frac{\chi_q(3)i}{3^2}-\frac{\chi_q(4)}{4^2}+\dots.$$
Now
 \begin{eqnarray*}\alpha \ge 1-\frac{1}{4^2}-\frac{1}{5^2}-\dots= 0.716\dots
\end{eqnarray*}
 and 
 \begin{eqnarray*}
|\beta|< \frac{1}{2^2}+\frac{1}{3^2}+\dots =0.64\dots
 \end{eqnarray*}
 Thus,
 $$f_q(1/5)> 0.6.$$


A couple of formulas may help us move forward here.
One is that if $\theta_1$ and $\theta_2$ are characters with coprime moduli $m_1$ and $m_2$ respectively, then (see [IK, (3.16)])
$$\tau( \theta_1 \theta_2) = \theta_1(m_2) \theta_2(m_1) \tau(\theta_1)\tau(\theta_2).$$
The other is that 
$$L(1-r,\theta)=-\frac{m^{r-1}}{r}\sum_{b=1}^m \theta(b) B_r(b/m)$$
for a character $\theta$ modulo $m$ and a positive integer $r$ where $B_r$ is the $r$th Bernoulli polynomial
(see [Wa,Theorem 4.2]). 
Recall the functional equation (see [D]) for a primitive character $\theta \bmod m$:
$$L(1-s,\theta) = \left(\frac{m}{2\pi}\right)^s  \Gamma(s) (e^{\pi i s/2}+\theta(-1) e^{-\pi i s/2}) L(s,\overline{\theta})/\tau(\overline{\theta})$$
It follows that for an even $\theta=\chi_q \psi$, with $q\equiv 3\bmod 4$ and $\psi$ an odd character modulo $p$, 
\begin{eqnarray*}L(2,\chi \overline{\psi})&=&-\pi \left(\frac{pq}{2\pi}\right)^{-1}L(-1, \theta)/\tau( \theta)\\
&=& -\pi \left(\frac{pq}{2\pi}\right)^{-1}L(-1,\chi\psi)/(\chi(p)\psi(q) \tau(\psi)i\sqrt{q})\\
\end{eqnarray*}
Therefore,
\begin{eqnarray*}
\Im\{ \tau(\psi)\overline{\psi}(a)L(2,\chi_q\overline{\psi})\}&=&\Re\{\frac{2\pi^2\chi_q(p)\overline{\psi}(aq)}{pq^{3/2} }L(-1,\chi_q\psi)\}
\\&=& -\Re\{\frac{\pi^2 \chi_q(p)\overline{\psi}(aq)}{\sqrt{q}}\sum_{b=1}^{pq} \chi_q(b)\psi(b)B_2(b/(pq))\}
\end{eqnarray*}
We sum this equation over the odd characters modulo $p$ using
\begin{eqnarray*}
\sum_{\psi \bmod p\atop \psi(-1)=-1}   \psi(\frac{b}{aq})&=&\frac 12  \sum_{\psi \bmod p }( \psi(\frac{b}{aq})- \psi(-\frac{b}{aq})) 
\\
&=& \frac{\phi(p)}{2} \left\{ \begin{array}{rl} 1& \mbox{if $b\equiv aq \bmod p$}\\-1 &\mbox{if $b\equiv -aq\bmod p$} \end{array} \right. 
\end{eqnarray*}
This gives 
\begin{eqnarray*}&&
\sum_{\psi \bmod p\atop \psi(-1)=-1} \Im\{ \tau(\psi)\overline{\psi}(a)L(2,\chi_q\overline{\psi})\}\\ &&\qquad 
=- \frac{\phi(p)}{2}  \frac{\pi^2  \chi_q(p)}{\sqrt{q}}\left(\sum_{b\le pq\atop b\equiv aq\bmod p}\chi_q(a)B_2(b/(pq))-\sum_{b\le pq\atop b\equiv -aq\bmod p}\chi_q(a)B_2(b/(pq))\right) 
\end{eqnarray*}
Note that 
$$B_2(x)=x^2-x+1/6.
$$
Also,
$$\sum_{b\le pq\atop b\equiv aq\bmod p}\chi_q(b)-\sum_{b\le pq\atop b\equiv -aq\bmod p}\chi_q(b)=0$$
and
$$\sum_{b\le pq\atop b\equiv aq\bmod p}b \chi_q(b)-\sum_{b\le pq\atop b\equiv -aq\bmod p}b \chi_q(b)=0.$$
Thus,
\begin{eqnarray*}&&
\sum_{\psi \bmod p\atop \psi(-1)=-1} \Im\{ \tau(\psi)\overline{\psi}(a)L(2,\chi_q\overline{\psi})\} 
=-\frac{\pi^2\chi_q(p)}{2 p^2 q^{5/2}}   \bigg(\sum_{b\le pq\atop b\equiv a q\bmod p}b^2\chi_q(b) -\sum_{b\le pq\atop b\equiv -aq\bmod p}b^2\chi_q(b) \bigg) .
\end{eqnarray*}
 Thus, we have
 \begin{theorem} For primes $p$ and $q$ both congruent to 3 modulo 4 and $1\le a<p/2$ we have
 \begin{eqnarray*}
 f_q(a/p)=-\frac{  \pi^2\chi_q(p)}{2 p^2 q^{5/2}}  \bigg(\sum_{b\le pq\atop b\equiv a q\bmod p}b^2\chi_q(b) -\sum_{b\le pq\atop b\equiv -aq\bmod p}b^2\chi_q(b) \bigg) .
 \end{eqnarray*}
 \end{theorem}
 and
 \begin{corollary} If
 \begin{eqnarray} \label{eqn:RH}
 \rm{test}_a(p,q):=-\chi_q(p) \bigg(\sum_{b\le pq\atop b\equiv aq\bmod p}b^2\chi_q(b) -\sum_{b\le pq\atop b\equiv -aq\bmod p}b^2\chi_q(b) \bigg) >0
 \end{eqnarray}
 for all $p<q$ which are primes congruent to 3 modulo 8 and all $0<a<p/2$, then the Riemann Hypothesis follows.
 \end{corollary}
 
 We note that by these techniques one can show
\begin{theorem}
 \begin{eqnarray*}
 f_q(a/q)=
 \frac{\pi^2}{2\sqrt{q}}\left(\chi_q(a)-\frac{1}{q^2}\sum_{c=1}^{q-1}c^2(\chi_q(c-a)-\chi_q(c+a))\right).
 \end{eqnarray*}
\end{theorem}

When this formula is  compared with our earlier formula
 \begin{eqnarray*}
  f_q(\frac a q) 
 =\frac{2\pi^2}{q^{3/2}}\left( \frac{a}{3}\sum_{n\le \frac{q-1}2} \chi_q(n) -\sum_{n=1}^a (a-n)\chi_q(n)\right)
 \end{eqnarray*}
we deduce the identity 
 $$ \frac{a}{3}\sum_{n\le \frac{q-1}2} \chi_q(n) -\sum_{n=1}^a (a-n)\chi_q(n)=\frac{q}{4}\left(\chi_q(a)-
\frac{1}{q^2} \sum_{c=1}^{q-1}c^2(\chi_q(c-a)-\chi_q(c+a))\right)
 $$
for $q\equiv 3 \bmod 4$.

 Now we indicate  another possible direction.
\begin{proposition} If
$$f_q(x)=0$$
then $x$ is a rational number.
\end{proposition}

\begin{proof}
By Corollary 1, $f_q(x)=0$ implies that $S_q(q/2)-S_q(qx)=0$. But $S_q(q/2)=h(q)$ is an integer. So $f_q(x)=0$ implies that
$S_q(qx)$ is a rational number. Now
$$S_q(qx)=\sum_{n\le [qx]} \chi_q(n)(1-\frac{n}{qx})=\sum_{n\le [qx]} \chi_q(n) -\frac{\sum_{n\le [qx]} {n}\chi_q(n)}{qx}.
$$
This has the shape $\mbox{integer} - \frac{\mbox{integer}}{qx}$ which can only be rational  if $x$ is a rational number.
\end{proof}
So, it suffices to show that $f_q(x)$ has no rational zeros; perhaps a congruence argument could work.
However, Theorem 5 is not much use here because the hypothetical  $x$ for which $f_q(x)=0$ would likely have a denominator that is divisible by $q$, so the conditions of Theorem 5 don't hold.

We remark that there are rational values of $x$ for which the numerator of $f_q(x)$    is congruent to 0 modulo $q$; for example
$$f_{19}(25/76)=\frac{19}{25} \qquad f_{19}(29/190)=\frac{19}{29} \qquad f_{19}(30/209)=\frac{19}{30}.$$
These examples, which all seem to have an $x$ with denominator divisible by $q$,  might be worth studying further.

  Here is one final formula that may or may not be useful.
Suppose that $f_q(x)=0$. Let $y=xq$. Then either
$$ \sum_{n\le y}\chi_q(n)=h(q) \qquad \mbox{and} \qquad \sum_{n\le y}n\chi_q(n)=0$$
or else  $y$ satisfies
\begin{eqnarray*} y=\frac{\sum_{n\le [y]}n\chi_q(n)}{\sum_{n\le [y] }\chi_q(n)-h(q)}.
\end{eqnarray*}
The first alternative seems unlikely as in that case there would be an interval on which $f_q(x)$ would be identically 0.

\section{Conclusion}
  Conjecture 1 has been checked for primes up to $10^9$ and it holds for those primes.  However, probabilistic grounds
  call into question it's truth for all primes $q\equiv 3 \bmod 8$. Of course, one only needs it's truth for a set of characters 
$\chi_q$ for which $\chi_q(n)=\lambda(n) $ for all $n\le N_q$ where $N_q\to \infty$ with $q$. 
Presumably something like this is correct (and should be  equivalent to RH), but it is not clear how to proceed.
But the results of section 6,  suggest a slightly alternative way forward which may have a more arithmetic flavor.

\enddocument 

The idea to proceed is to show that
the numerator of 
$$W(q,x):=h(q)-\sum_{n\le qx} \chi_q(n)(1-\frac{n}{qx} ) $$
is not  congruent to 0 modulo $q$ for every rational $x< 1/2$. 
If this is true then RH follows. We conjecture that this is correct.

\begin{conjecture}If $0<x<1/2$ is a rational number, then for any prime  $q\equiv 3 \bmod 8$, the numerator of
\begin{eqnarray*}   W(q,x)  
\end{eqnarray*} is not divisible by $q$.
\end{conjecture}

There is an interesting statistical phenomenon that if one computes the values
$\rm{fac}(a) W(q,a/r)$ for $(a,r)=1$, then reduces these  modulo $q$ and arranges them form smallest to largest,
one finds that the first few values in this list are divisible by the class number $h(q)$. This might suggest that there is 
something  arithmetic going on here which prevents $W(q,x)$ from being divisible by $q$. So, maybe there is some hope to prove Conjecture 2.

 It seems that $\mbox{test}(p,q)=\mbox{test}_1(p,q)$ is an integer with some interesting properties. For example, occasionally $\mbox{test}(p,q)$ is divisible by $q$:
 $$\mbox{test}(1163, 3511)=2^5\cdot5 \cdot 3511 \qquad \mbox{test}(919, 3271)=2^3\cdot13\cdot 3271 \qquad \mbox{test}(719,2971)=2^2\cdot11\cdot 2971$$
  $$\mbox{test}(2087,6703)=2^3\cdot11 \cdot 6703 \qquad \mbox{test}(1427, 5431)=2^5\cdot7\cdot 5431 \qquad \mbox{test}(919,3671)=2^3\cdot3\cdot7\cdot 3671$$
 Also, $\mbox{test}(2851,4951)=2^2\cdot 151381$ and $151381 \equiv 2851 \bmod 4951$ and 
 $\mbox{test}(2851,4079)=2^2\cdot 345487$ and $345487\equiv 2851 \bmod 4079$ 

Same for (2351,2371), (2087,2099), (1847, 3491), (1847,3069), (719,1063)

\begin{conjecture} We conjecture that (\ref{eqn:RH}) holds for all primes $p$ and $q$ which are congruent to 3 modulo 8 for which $0<2a<p<q$ where $\chi_q(p)=\left(\frac p q\right)$
is the Legendre symbol modulo $q$.
\end{conjecture}

\begin{proposition} If $aq<p$, then 
\begin{eqnarray*}
\mbox{test}(a,p,q)= 4 a \cdot h(q)
\end{eqnarray*}
where $h(q)$ is the class number of $\mathbf Q(\sqrt{-q})$.
\end{proposition}

We conjecture  that when $q$ is a prime congruent to 7 mod 8 that the maximum on the left hand-side of (\ref{eqn:inequality}) is achieved at $N=q/4$, and (as already conjectured for these $q$) that 
$\sum_{n\leqslant q/4} n\chi_q(n) >0$; these two conjectures together imply the Riemann Hypothesis.

 $$ \begin{tabular}{|c|c|c|c|c|c|c|c|c|c|c|c|c|c|c|c|c|c|}
  \hline
  n&1&2&3&4&5&6&7&8&9&10&11&12&13&14&15&16&17\\
  \hline
  $\lambda(n)$&1&-1&-1&1&-1&1&-1&-1&1&1&-1&-1&-1&1&1&1&-1\\
  \hline
  $\chi_q(n)$&1&-1&-1&1&-1&1&-1&-1&1&1&-1&-1&-1&1&1&1&-1\\
  \hline
  \end{tabular}
  $$

   \section{$f(a/q)$}
 We can express the value of $f$ at a rational number in terms of an average over Dirichlet L-functions at 2.  To do this we use
 \begin{lemma} If $(a,q)=1$ then 
 $$ e(a/q)=\frac{1}{\phi(q)} \sum_{\chi \bmod q} \tau(\overline{\chi}) \chi(a) $$
 \end{lemma}
 This leads to
 \begin{eqnarray*}
 f(a/q)&=&\Im  \frac{1}{\phi(q)} \sum_{\chi \bmod q} \tau(\overline{\chi}) \chi(a) 
 \sum_{n=1}^\infty \frac{\lambda(n)\chi(n) }{n^2}   \\
 &=&\Im  \frac{1}{\phi(q)} \sum_{\chi \bmod q} \tau(\overline{\chi}) \chi(a) \frac{L(4,\chi)}{L(2,\chi)}
 \end{eqnarray*}

Here is another example with  the tail end of the graph for $q=2647$:
\begin{center}
 \includegraphics[scale=0.9]{gr2.eps}
\end{center}
In this example with $h(2647)=15$, the inequality of Theorem 5 holds for $N\leqslant 2647/4$ as desired but you can see that it fails 
as $N/2647$ gets close to 1/2.

 We may examine Theorem 5 with $N=q/4$. It is interesting to consider the case where  $q\equiv 7 \bmod 8$, since then it is the case that $h(q)=\sum_{n\leqslant q/4} \chi_q(n)=h(q)$. So, if 
 $$T(q):=\sum_{n\leqslant q/4} n \chi_q(n)>0$$
 then the desired inequality of Theorem 5 holds (when $N=q/4$). Here is a list of values of $T(q)$ for primes $q\equiv 7 \bmod 8$:
 \begin{eqnarray*}&&
 1, 5, 10, 14, 29, 42, 57, 80, 111, 91, 130, 165, 180, 178, 191, 258, \
259, 327, 382, 322, 426, 533,\\ &&
562, 480, 601, 481, 723, 798, 965, 723, \
912, 1071, 885, 1254, 1339, 1125, 1142, 1125, 1340, \\&&1691, 1669, 1313, \
1906, 1599, 1982, 1855, 2024, 1656, 1796, 1851, 2500, 2680, 2663, \
2181, \\&&2720, 2247, 2991, 3002, 2668, 2941, 3374, 2432, 2775, 3329, \
2869, 3342, 2828, 2903, 3756, \\&&4468, 4085, 3356, 4753, 3405, 4002, \
4789, 5076, 5390, 4495, 4024, 4235, 4618, 5580, 4572, 6164,\\&& 5888,
6809, 5164, 6176, 5492, 5379, 5494, 6905, 5688, 7525, 6374,\dots
 \end{eqnarray*}
 We conjecture that $T(q)>0$ for all  primes $q\equiv 7 \bmod 8$.
 Here is a plot of the values
 \begin{center}
 \includegraphics[scale=0.9]{gr3.eps}
\end{center}

 In trying to construct  a counterexample to Conjecture 1, we replace the $\chi_q(n)$ with $q\equiv 3 \bmod 4$ with a completely multiplicative function $c(n)$ for which $c(2)=-1$ and $c(p)=1$ for all primes $p>2$. For this function we consider  
 $$ F_q(x):=x(T(q/2)-T(q x)) $$
 for various $q$ 
 where $$T(x)=\sum_{n\le x} c(n)(1-\frac {n}{x}).$$
 Here is a plot of $F_{241}(x)$ for $0\le x\le 1/2$:
 \begin{center}
  \includegraphics[scale=0.9]{gr241.eps}
\end{center}
 This may or may not be evidence for Conjecture 1.

  \section{Explicit Estimates}
 Here we gather some material with the goal of proving the converse to Theorem 1. 
 In this section we assume the Riemann Hypothesis throughout.
 
 The first estimate we need to make explicit is (see [MV1], Lemma 12.1] or [T, Theorem 9.6A])
 $$\frac{\zeta'}{\zeta}(s)=\frac{-1}{s-1}+\sum_{|\gamma-t|\leqslant 1} \frac{1}{s-\rho}+O(\log \tau)
 $$
 uniformly for $-1\leqslant \sigma \leqslant  2$ where $\tau=|t|+4$.
 
 This follows from Corollary 10.14 of [MV1]:
 \begin{eqnarray*}
 \frac{\zeta'}{\zeta}(s) =B +\frac 12 \log \pi -\frac{1}{s-1} -\frac 12 \frac{\Gamma'}{\Gamma}(s/2+1)+\sum_\rho\left(\frac{1}{s-\rho}+\frac 1 \rho\right)
 \end{eqnarray*}
 where 
 $$B=-0.0230957\dots.$$
 
 We have a second formula for $\frac{\zeta'}{\zeta}$ (see MV1, (13.35)]:
 \begin{eqnarray*}
 \frac{\zeta'}{\zeta}(s)= -\sum_{n\le xy}w(n)\frac{\Lambda(n)}{n^s}+\frac{(xy)^{1-s}-x^{1-s}}{(1-s)^2 \log y}
 -\sum_\rho\frac{(xy)^{\rho-s}-x^{\rho-s}}{(\rho-s)^2 \log y}
 -\sum_{k=1}^\infty \frac{(xy)^{-2k-s}-x^{-2k-s}}{(2k+s)^2 \log y}
\end{eqnarray*}

\end{document}